\documentclass[letterpaper,11pt,reqno]{amsart}
\usepackage[dvipdfmx]{graphicx}
\usepackage[T1]{fontenc}
\usepackage[utf8]{inputenc}
\usepackage{lmodern}
\usepackage[english]{babel}
\usepackage{amsmath,a4wide}
\usepackage{xfrac}
\usepackage{esint}
\usepackage{stmaryrd,mathrsfs,bm,amsthm,mathtools,yfonts,amssymb,braket,booktabs,graphics,amsfonts}
\usepackage{latexsym,microtype,indentfirst,hyperref}
\usepackage{xcolor}
\usepackage{courier}

\newtheorem{theorem}{Theorem}[section]

\newtheorem{lemma}[theorem]{Lemma}

\theoremstyle{definition}

\newtheorem{remark}[theorem]{Remark}

\numberwithin{equation}{section}
\numberwithin{subsection}{section}

\newcommand{\R}{\mathbb{R}} 

\newcommand{\eps}{\varepsilon} 

\begin{document}
\title[Brakke's formulation of velocity]{Brakke's 
formulation of velocity and \\ the second order regularity property}
\author[R. Mori]{Ryunosuke Mori}
\address{Department of Mathematics, Tokyo Institute of Technology, 2-12-1 Ookayama, Meguro-ku, Tokyo 152-8551, Japan}
\email{45c136045@gmail.com}
\author[E. Tomimatsu]{Eita Tomimatsu}
\address{Department of Mathematics, Tokyo Institute of Technology, 2-12-1 Ookayama, Meguro-ku, Tokyo 152-8551, Japan}
\email{tomimatsu.e.aa@m.titech.ac.jp}
\author[Y. Tonegawa]{Yoshihiro Tonegawa}
\address{Department of Mathematics, Tokyo Institute of Technology, 2-12-1 Ookayama, Meguro-ku, Tokyo 152-8551, Japan}
\email{tonegawa@math.titech.ac.jp}

\thanks{Y.~Tonegawa is partially supported by JSPS Grant 18H03670, 19H00639 and 17H01092.
The authors thank the anonymous referee for reading the manuscript carefully.} 

\begin{abstract}
Suppose that a family of $k$-dimensional surfaces in $\mathbb R^n$ evolves by the motion law of $v=h+u^\perp$ in
the sense of Brakke's 
formulation of velocity,
where $v$ is the normal velocity vector, $h$ is the generalized mean curvature
vector and $u^\perp$ is the normal projection of a given vector field $u$
in a dimensionally sharp integrability class. When the flow is locally close to 
a time-independent $k$-dimensional plane in a weak sense of measure in space-time, it is 
represented as a graph of a $C^{1,\alpha}$ function over the plane. On the other hand,
it is not known if the graph satisfies the PDE of $v=h+u^\perp$ pointwise
in general. For this problem, when $k=n-1$ and under the additional assumption that the distributional
time derivative of the graph is a signed Radon measure, it is proved that
the graph satisfies the PDE pointwise. An application to a short-time 
existence theorem for a surface evolution problem is given. 
\end{abstract}

\maketitle

\section{Introduction}
A family of $k$-dimensional surfaces $\{M_t\}_{t\geq 0}$ in $\mathbb R^n$ is called the mean curvature flow (abbreviated hereafter as MCF) if the 
velocity $v$ is equal to the mean curvature vector $h$ of $M_t$ at each point on $M_t$. Given a smooth compact surface $M_0$, the MCF as the 
initial value problem is well-posed until some singularities such as vanishing and pinching appear. To extend the flow
past singularities, a number of generalized formulations of MCF have been proposed: we mention among others, the level-set flow \cite{CGG,ES}, Brakke flow \cite{Br78}
and BV solution \cite{Luckhaus}. The properties of these generalized MCFs and their relations have been studied by numerous
researchers for the last 40 years or so. The present paper is concerned with a subtle aspect on the formulation of velocity in the definition 
of the Brakke flow. Within this framework, the velocity of the flow is characterized by the so-called Brakke's inequality which dictates
the rate of change of surface measure. We may, for example, characterize a 
normal vector field $v$ to be the velocity of $M_t$ in the sense of Brakke if 
\begin{equation}\label{nor1}
\frac{d}{dt}\Big(\int_{M_t}\phi\,d\mathcal H^k\Big)\leq \int_{M_t}\Big\{(\nabla\phi-\phi\,h)\cdot v+\partial_t\phi\Big\}\,d\mathcal H^k
\ \ \mbox{for all non-negative }\phi=\phi(x,t).\end{equation}
Here, the symbol $\mathcal H^k$ is the $k$-dimensional Hausdorff measure and \eqref{nor1} is understood in the sense of distributions. 
If $\{M_t\}_{t\geq 0}$ is
any smooth family of surfaces (which need not be MCF), one can prove that 
the inequality \eqref{nor1} is satisfied if and only if the normal vector field $v$ is the usual velocity of $M_t$ (see \cite[Chapter 2]{Ton-b}). 
Omitting all the details on what $\{M_t\}_{t\geq 0}$ should additionally satisfy, if we can take $v=h$ in \eqref{nor1}, 
then the family $\{M_t\}_{t\geq 0}$ may be roughly called a MCF in the sense of Brakke (or Brakke flow). 

A special feature of the Brakke flow compared to other formulations is Brakke's local regularity theorem  \cite{Br78,Kasai-Tone,Ton-2}, namely,
if a Brakke flow is locally close to a time-independent $k$-dimensional plane in a weak sense of measure, 
then it is a $C^\infty$ MCF
in the space-time interior. The a priori regularity of the Brakke flow is just rectifiability for almost all
time with square integrable generalized mean curvature, so it is highly nontrivial to prove such a high degree of 
regularity. The regularity theorem has two parts, the $C^{1,\alpha}$ part \cite{Kasai-Tone}
which shows the flow is represented as a $k$-dimensional graph of $C^{1,\alpha}$ function among other things, and the
$C^{2,\alpha}$ part \cite{Ton-2} which shows the flow satisfies $v=h$ pointwise. At this point, the standard 
regularity theory for parabolic PDE can be applied to obtain $C^{\infty}$ regularity. Just as usual for parabolic problems, 
the regularity with respect to the time variable is halved, that is, when we talk about $C^{2,\alpha}$, it means
that the time derivative of the graph is $\alpha/2$-H\"{o}lder continuous in the $t$ direction, so at least, 
the classical pointwise notions of velocity and mean curvature are well-defined once the $C^{2,\alpha}$ regularity is available. 
The regularity theorem of
\cite{Kasai-Tone,Ton-2} is actually more general than Brakke's original version in that the motion law can be
$v=h+u^{\perp}$ with a given vector field $u$ in a suitable regularity class. Here, $u^\perp$ is the projection of
$u$ to the orthogonal complement of the tangent space of $M_t$. Just to familiarize the reader, note that the motion law is 
a geometric analogue of the heat equation with inhomogeneous term: ``$\partial_t f=\Delta f+u$''. To relate to the time-independent case, 
the result \cite{Kasai-Tone} is the precise parabolic extension of the Allard
regularity theorem \cite{Allard} and gives a $C^{1,\alpha}$ regularity under the dimensionally sharp 
integrability assumption on $u$. If $u$ is assumed to be H\"{o}lder continuous, then \cite{Ton-2} proves 
the $C^{2,\alpha}$ regularity with the conclusion that $v=h+u^{\perp}$ holds classically in the space-time interior.
 
One question not addressed in \cite{Kasai-Tone} and the main subject of the present paper may be explained as follows. 
In the situation considered in \cite{Kasai-Tone} where $u$
satisfies
\begin{equation}\label{power}
\|u\|_{L^{p,q}}:=\Big(\int_0^T\Big(\int_{M_t}|u(x,t)|^p\,d\mathcal H^k(x)\Big)^{q/p}\,dt\Big)^{1/q}<\infty
\end{equation} 
with $p,q\in [2,\infty)$ satisfying
\begin{equation}\label{power2}
\alpha:=1-\frac{k}{p}-\frac{2}{q}>0,
\end{equation}
if $M_t$ is locally close to a $k$-plane in measure and $v=h+u^\perp$ is satisfied in the sense of Brakke as in
\eqref{nor1}, then \cite{Kasai-Tone} proves that 
$M_t$ is represented as a $C^{1,\alpha}$ graph with a suitable estimate (see \cite[Theorem 8.7]{Kasai-Tone}). 
It is then reasonable to ask if the
graph belongs additionally to $L^q((0,T);W^{2,p})\cap W^{1,q}((0,T);L^p)$, a natural class in view of the 
parabolic PDE regularity theory. Furthermore, one may ask
if the motion law of $v=h+u^{\perp}$ is satisfied 
almost everywhere pointwise. In other words, the question is: is the $C^{1,\alpha}$ graph of $M_t$ a strong PDE solution 
for $v=h+u^\perp$? Though the affirmative conclusion may sound reasonable, 
the answer is not known in general at present. A quick remark on the reason is 
that the formulation \eqref{nor1} is based on the inequality and fails to give a PDE with equality to work with even when
$M_t$ is a $C^{1,\alpha}$ graph. 
The comparison with the time-independent situation of the Allard regularity theorem is interesting
in that, once $C^{1,\alpha}$ regularity is established for the weak formulation of $0=h+u^\perp$ (see \cite[10.2]{Kasai-Tone}), where \eqref{power}
and \eqref{power2} in the time-independent situation
require $p>k$ (with $q$ arbitrary) and $h=-u^\perp\in L^p$, then the standard elliptic PDE regularity theory 
shows that the graph is in $W^{2,p}$ and a strong solution for $h=-u^\perp$. Thus, showing 
$C^{1,\alpha}$ leads automatically to the strong solution in the time-independent case of Allard regularity theory. As stated already, if $u$ is in $C^{\alpha}$, then \cite{Ton-2} shows that the graph is a 
classical $C^{2,\alpha}$ solution satisfying the PDE, thus there is a certain gap between the H\"{o}lder case and $L^{p,q}$ case presently. 

Towards this question in this paper, for the hypersurface case $k=n-1$, we prove that $M_t$ is a strong PDE solution of $v=h+u^\perp$
if we additionally assume a certain regularity with respect to $t$:
\begin{theorem}\label{main-brief}
Suppose that $M_t$ is represented as a $C^{1,\alpha}$ graph $x_n=f(x_1,\ldots,x_{n-1},t)$ locally in space-time, and
\eqref{nor1} is satisfied in the sense of distributions with $v=h+u^\perp$ for $u$ satisfying \eqref{power} and \eqref{power2}. 
If the time derivative $\partial_t f$ exists as a signed Radon measure, 
then $\partial_t f$ and $\nabla^2 f$ belong to $L^{p,q}$ 
and the function $f$ satisfies $v=h+u^\perp$ pointwise for almost all point, that is, $f$ is a 
strong solution of the PDE, $v=h+u^\perp$. 
\end{theorem}
The precise statement with detailed descriptions on $M_t$ will be given later, but the reader may think that
$\{M_t\}_{t\geq 0}$ is a ``Brakke-like flow'' satisfying \eqref{nor1} with $v=h+u^\perp$.
The result shows that Brakke's formulation of the velocity $v=h+u^\perp$ gives a strong solution 
if it is supplemented by the additional assumption on the time derivative. 
Though this additional assumption itself does not appear to follow from Brakke's formulation of \eqref{nor1},
it is interesting to note that the flow obtained as a limit of the Allen-Cahn equation with
transport term 
in \cite{TaTo16} satisfies $\partial_t f\in L^2$ in addition to \eqref{nor1}, so in particular a signed Radon measure. 
Roughly speaking, $\partial_t f\in L^2$
follows from the property that the distributional derivative of the limit phase function with respect to $t$ 
is $L^2$ with respect to the surface measure of interface. The latter property is strongly related to the MCF in the sense of
BV solutions \cite{Luckhaus}, where the existence of $L^2$ velocity is a part of the definition. 
Combined with \cite{TaTo16}, we prove a local-in-time existence of the strong solution for the
surface evolution problem $v=h+u^\perp$
where $u$ is a given vector field in a Sobolev space whose trace on $\{M_t\}_{t\geq 0}$
is $L^{p,q}$, see Theorem \ref{th:existence}.  

For Theorem \ref{main-brief},
we first give a proof under a stronger assumption of 
$\partial_t f\in L^2$. There are three reasons for doing this: (1) The proof is 
simpler compared to the case of $\partial_t f$ being a signed Radon measure. (2) The proof should work in principle for general $k$-dimensional case. 
(3) The application to the Allen-Cahn equation with
transport term falls in this situation and it is good to have a simpler 
proof in this case separately. If $\partial_t f\in L^2$,
the graph can be approximated by a smooth function and some appropriate
scaling argument shows the desired PDE. For the case of 
signed Radon measure, we use the fact that Brakke's inequality
gives rise to a Radon measure with which 
Brakke's inequality is turned into the \emph{equality}. 
Then using the signed distance function to the suitably mollified smooth 
graph and estimating the errors coming from the approximation, 
we show that the graph has a weak $L^2$ $t$-derivative. We remark
that, because of
the use of signed distance function, the proof seems to be limited to the 
hypersurface case. One natural question is that whether the 
additional assumption on $\partial_t f$ is necessary or not and
presently we do not know the answer. 
 
The organization of the paper is as follows. In Section \ref{results}, 
detailed assumptions and main results are presented. In Section \ref{ver1},
the proof of the main regularity theorem under the stronger assumption of $\partial_t f\in
L^2$ is given, and that of general case is given in Section \ref{ver2}.
In the final Section \ref{PDE}, the proof of 
Theorem \ref{th:existence} is given. 

\section{main results}\label{results}
\subsection{Basic notation}
For $0<r<\infty$ and $a\in\mathbb R^k$ ($1\leq k\leq n$), we define $B_r^k(a):=\{x\in\mathbb R^{k}\,:\,|x-a|<r\}$ and $B_r^k:=B_r^k(0)$.
We often use $B_r^{n-1}$ throughout the paper, so we write $B_r$ for $B_r^{n-1}$. The symbols $\mathcal L^k$ 
defined in $\mathbb R^k$ and
$\mathcal H^k$ defined in $\mathbb R^n$ are the
$k$-dimensional Lebesgue measure and the Hausdorff measure, respectively. Notation for the functional spaces
such as $L^p(B_r)$ and $W^{k,p}(B_r)$ are the same as in \cite{GiTr83}. 

\subsection{Setting of the problem} 
Suppose that a function $f\,:\,(x,t)\in B_1\times(0,1)\rightarrow\mathbb R$ is a $C^{1,\alpha}$ function in the parabolic sense,
where 
$\alpha\in(0,1)$ is defined as in \eqref{power2} with given $p,q\in[2,\infty)$. Here, parabolic $C^{1,\alpha}$ means that
\begin{equation}\label{lip}
\begin{split}
[f]_{C^{1,\alpha}}:=&\sup_{(y_j,s_j)\in B_1\times(0,1),\,j=1,2}\frac{|\nabla f(y_1,s_1)-\nabla f(y_2,s_2)|}{\max\{|y_1-y_2|^\alpha,
|s_1-s_2|^{\alpha/2}\}}\\ &+\sup_{(y,s_j)\in B_1\times(0,1),\,j=1,2}\frac{|f(y,s_1)-f(y,s_2)|}{|s_1-s_2|^{(1+\alpha)/2}}<\infty.
\end{split}
\end{equation}
The symbol $\nabla f$ is the gradient of $f$ with respect to the space variables. 
With this $f$, define the $(n-1)$-dimensional hypersurface $M_t$ by 
\begin{equation}\label{mdef}
M_t:=\{(x,f(x,t))\,:\, x\in B_1\}\subset \mathbb R^n
\end{equation}
for $t\in(0,1)$. We assume that $M_t$ has the generalized mean curvature vector (see \cite{Allard,Simon} for
the definition) $h=h(\cdot,t)$ for $\mathcal L^1$-a.e.~$t\in(0,1)$. In the case that $M_t$ is represented as a graph, we may consider either that $h$ is defined 
on $M_t$ or $B_1$, and we may use the same notation $h$ with no fear of confusion. From the definition of the
generalized mean curvature vector and since $h=(h\cdot\nu)\nu$ by the perpendicularity theorem of Brakke \cite[Chapter 5]{Br78}, $h$
satisfies
\begin{equation}\label{hdef1}
\int_{B_1}\nabla\psi\cdot\Big(\frac{\nabla f}{\sqrt{1+|\nabla f|^2}}\Big)\,dx=-\int_{B_1}\psi\,h\cdot\nu\,dx
\end{equation}
for all $\psi\in C^1_c(B_1)$.  Here, $\nu:=(-\nabla f,1)/\sqrt{1+|\nabla f|^2}$ is the unit normal vector of $M_t$. 
We assume that $h$ is in $L^2$, namely,
\begin{equation}\label{hdef2}
\int_0^{1}\int_{M_t}|h|^2\,d\mathcal H^{n-1}dt=\int_0^{1}\int_{B_1}|h|^2 \sqrt{1+|\nabla f|^2}\,dxdt<\infty.
\end{equation}
Next, suppose that a vector field $u(\cdot,t)\,:\,M_t\rightarrow \mathbb R^n$ defined for $\mathcal L^1$-a.e.~$t\in(0,1)$
satisfies (just like $h$, we may use the same notation $u(\cdot,t)$ as a function defined on $M_t$ or $B_1$)
\begin{equation}\label{hdef3}
\int_0^{1}\Big(\int_{M_t}|u|^p\,d\mathcal H^{n-1}\Big)^{\frac{q}{p}}\,dt=\int_0^{1}\Big(\int_{B_1}|u|^p
\sqrt{1+|\nabla f|^2}\,dx\Big)^{\frac{q}{p}}\,dt<\infty.
\end{equation}
For all 
non-negative test function with compact support 
$\phi\in C_c^1((B_1\times\mathbb R)\times(0,1);\mathbb R^+)$, assume that we have the inequality
\begin{equation}\label{Bineq}
0\leq \int_0^{1}\int_{M_t}(\nabla\phi-\phi\, h)\cdot\{h+(u\cdot\nu)\nu\}+\partial_t\phi\,d\mathcal H^{n-1}dt.
\end{equation}
Here note that $\nabla\phi$ is the gradient of $\phi$ in $\mathbb R^n$ and the formula \eqref{Bineq} corresponds
to Brakke's formulation for ``$v=h+u^\perp$'' as in \eqref{nor1}. 
\begin{remark}
The above assumptions are locally satisfied after a suitable change of variables 
once the regularity result in \cite{Kasai-Tone} is applied
under the assumptions \cite[(A1)-(A4)]{Kasai-Tone}, see the 
statement of 
\cite[Theorem 8.7]{Kasai-Tone}. 
\end{remark}
\label{setting}
\subsection{Statement of main result}
The following is the main theorem of the present paper. 
\begin{theorem}\label{main1}
Suppose that $f$, $h$ and $u$ are as discussed in Section \ref{setting} and assume additionally
that $\partial_t f$ is a signed Radon measure on $B_1\times (0,1)$. 
Then, we have 
\begin{equation}\label{m2}
f\in L^q((s^2,1);W^{1,p}(B_{1-s}))\cap W^{1,q}((s^2,1);L^p(B_{1-s}))
\end{equation}
for all $s\in(0,1)$
and $f$ satisfies the motion law of $v=h+u^\perp$, that is, 
\begin{equation}\label{m1}
\frac{\partial_t f}{\sqrt{1+|\nabla f|^2}}={\rm div}\Big(\frac{\nabla f}{\sqrt{1+|\nabla f|^2}}\Big)
+u\cdot\frac{(-\nabla f,1)}{\sqrt{1+|\nabla f|^2}}
\end{equation}
$\mathcal L^n$-a.e.~on $B_1\times(0,1)$. 
\end{theorem}

We give one application of Theorem \ref{main1} to an existence theorem 
of surface evolution problem studied in \cite{TaTo16}. The assumptions on $u$ 
are the same as 
\cite[Theorem 2.2]{TaTo16} but to avoid confusion, $p,q$ in \cite{TaTo16} are denoted 
by $\beta,\gamma$. The claim is that, whenever the $C^{1,\alpha}$
regularity theorem of \cite{Kasai-Tone} is applied to the solution established in \cite{TaTo16}
in some space-time neighborhood, then
it is a strong solution in the same neighborhood. Since the paper \cite{TaTo16} shows the short-time existence of solution 
for which the $C^{1,\alpha}$ regularity theorem is applicable for all points, we obtain the following.  
\begin{theorem}\label{th:existence}Suppose $n\ge 2$,
\[2<\gamma<\infty,\ \ \frac{n\gamma}{2(\gamma-1)}<\beta<\infty\ \ \left(\frac{4}{3}\le \beta\ {\rm in\ addition\ if}\ n=2\right)\]
and $\Omega=\R^n$ or $\mathbb{T}^n$. Given any time-dependent Sobolev vector field
\[u\in L^\gamma_{loc}([0,\infty);(W^{1,\beta}(\Omega))^n)\]
and a non-empty bounded domain $\Omega_0\subset\Omega$ with $C^1$ boundary $M_0=\partial\Omega_0$, there exist $T>0$ and a family of $C^{1,\alpha}$ hypersufaces $\{M_t\}_{t\in(0,T)}$ whose motion law is 
$v=h+u^\perp$ as a strong solution and $\lim_{t\rightarrow 0+}M_t=M_0$ in $C^1$ topology.
Here, 
$\alpha=2-n/\beta-2/\gamma$ if $\beta<n$, and if $\beta\geq n$ one may 
take any $\alpha$ with $0<\alpha<1-2/\gamma$. In addition,
the vector field $u$ is defined as a trace on $M_t$ for a.e.~$t\in(0,T)$ 
and $\|u\|_{L^{p,q}}$ (as in \eqref{power})
is finite with $p=\beta(n-1)/(n-\beta)$ and $q=\gamma$ if $\beta<n$. 
\end{theorem}
To be clear about being a strong solution here, for each $t\in(0,T)$ and $x\in M_t$, 
there exists a space-time neighborhood in which $\cup_{t\in(0,T)}(M_t\times\{t\})$ is represented 
as a graph of a function $f$ with the regularity of \eqref{m2} and satisfying the equation \eqref{m1}
for $\mathcal L^n$-a.e. in the neighborhood. Note that the conditions on $\beta$ and $\gamma$
imply \eqref{power} for $p$ and $q$. If $\beta>n$, then the Sobolev
embedding shows $u\in L_{loc}^\gamma([0,\infty);(C^{1-\frac{n}{\beta}}(\Omega))^n)$ and $\|u\|_{L^{p,q}}<\infty$ for any $p>2$ and $q=\gamma$
and the corresponding regularity result follows for $f$ (the last part is also true for the case of $\beta=n$). 
\section{Proof of Theorem \ref{main1}: $\partial_t f\in L^2$ case}\label{ver1}
First we note that \eqref{hdef1} combined with the standard argument (for example, see
\cite[Section 6.3.1]{Ev10}) that 
\begin{equation}
\begin{split}
&{\rm div}\Big(\frac{\nabla f}{\sqrt{1+|\nabla f|^2}}\Big)=h\cdot\nu \,\,\mbox{ a.e.~in $B_1$ and } \\ &
 \|\nabla^2 f(\cdot,t)\|_{L^2(B_{1-s})}\leq C(\|f(\cdot, t)\|_{L^2(B_1)}+\|h(\cdot,t)\|_{L^2(B_1)})
\end{split}
\end{equation}
for $\mathcal L^1$-a.e.~$t\in(0,1)$ and $s\in (0,1)$ for $C=C(n,\|\nabla f\|_{C^{0}},s)$
and thus, combined with \eqref{hdef2}, we have
\begin{equation}\label{w22}
f\in L^2((0,1);W^{2,2}(B_{1-s})) \ \ \ \mbox{for all }s\in (0,1).
\end{equation}

For the rest of this section, assume the setting explained in Section \ref{setting}. 
\begin{lemma} 
Suppose that $\partial_t f\in L^2(B_1\times(0,1))$. Then for all $\psi\in C_c^1(B_1\times
\mathbb R\times(0,1))$, we have
\begin{equation}\label{a-vel}
\int_0^{1}\int_{M_t}\{(\nabla\psi-
\psi h)\cdot v+\partial_t\psi\}\,d\mathcal H^{n-1}dt=0,
\end{equation}
where $\nabla\psi$ is the gradient of $\psi$ in $\mathbb R^n$ and $v=\frac{\partial_t f}{\sqrt{1+|\nabla f|^2}}\nu$. 
\end{lemma}
\begin{proof}
In the following calculations, we write the gradient of $\psi$ with respect to
the first $n-1$ variables as $\nabla' \psi$ and the derivative with respect to $x_n$ as $\partial_{x_n}\psi$.
First we assume that $f$ is in $C^{\infty}(B_1\times(0,1))$ and $M_t$ is defined as in \eqref{mdef}. Then, the direct computation shows
\begin{align*}
    \frac{d}{dt} \int_{M_t} \psi\,d\mathcal H^{n-1}&=\frac{d}{dt} \int_{B_1} \psi(x,f(x,t),t) \sqrt{1+\lvert \nabla f \rvert^2} \, dx
    \\&=\int_{B_1} ( \partial_t \psi +\partial _{x_n} \psi \,\partial_t f ) \sqrt{1+\lvert \nabla f \rvert^2}\, dx
    + \int_{B_1} \psi \frac{\nabla f \cdot \nabla \partial_t f}{\sqrt{1+\lvert \nabla f \rvert^2}} \, dx.
\end{align*}
Integrating by part, the second term is rewritten as follows.
{\small
\begin{align*}
    \int_{B_1} \psi \frac{\nabla f \cdot \nabla\partial_t f}{\sqrt{1+\lvert \nabla f \rvert^2}} \, dx =&-\int_{B_1}\partial_t f\,\mathrm{div}\left( \psi \frac{\nabla f}{\sqrt{1+\lvert \nabla f \rvert^2}} \right)\, dx=-\int_{B_1}\psi \,\partial_t f \,\mathrm{div}\left( \frac{\nabla f}{\sqrt{1+\lvert \nabla f \rvert^2}} \right)\, dx \\ &-\int_{B_1}\partial_t f \frac{\nabla' \psi \cdot \nabla f}{\sqrt{1+\lvert \nabla f \rvert^2}} \, dx-\int_{B_1}\partial_t f  \frac{\partial_{x_n}\psi \lvert\nabla f \rvert^2}{\sqrt{1+\lvert \nabla f \rvert^2}} \, dx.
\end{align*}
}
Thus
\begin{align*}
    \frac{d}{dt} \int_{M_t} \psi \, d\mathcal H^{n-1} &
    =\int_{B_1} \partial_t \psi \sqrt{1+\lvert \nabla f \rvert^2}\, dx
    + \int_{B_1}\partial_t f  \frac{\partial_{x_n}\psi }{\sqrt{1+\lvert \nabla f \rvert^2}} \, dx\\&\indent-\int_{B_1}\psi\,\partial_t f \,\mathrm{div}\left( \frac{\nabla f}{\sqrt{1+\lvert \nabla f \rvert^2}} \right)\, dx-\int_{B_1}\partial_t f \frac{\nabla' \psi \cdot \nabla f}{\sqrt{1+\lvert \nabla f \rvert^2}} \, dx\\& =\int_{B_1} \partial_t \psi \sqrt{1+\lvert \nabla f \rvert^2}\, dx
    +\int_{B_1} (\nabla \psi - \psi h)\cdot v \sqrt{1+\lvert \nabla f \rvert^2} \, dx,
\end{align*}
where $h={\rm div}(\frac{\nabla f}{\sqrt{1+|\nabla f|^2}})\nu$, $v=\frac{ \partial_t f}{\sqrt{1+\lvert \nabla f \rvert^2}}\nu$ with $\nu=\frac{1}{\sqrt{1+|\nabla f|^2}}(-\nabla f,1)$. Integrating over $t\in (0,1)$, we obtain \eqref{a-vel}.

For general case, we have $\nabla^2 f,\,\partial_t f \in L^2_{loc}(B_1\times(0,1))$, the
first one by \eqref{w22} and the second one from the assumption. 
Then we can take a sequence $\{f_k\}$ of smooth functions such that, as $k\rightarrow\infty$,
\begin{align*}
   & f_k\rightarrow f,\ \nabla f_k\rightarrow \nabla f\ {\rm in}\ C^{\alpha^{\prime}/2}_{loc}(B_1\times(0,1))\quad({\rm for\ any}\ \alpha'\in(0,\alpha)),\\&\nabla^2f_k\rightarrow \nabla^2 f,\ \partial_t f_k\rightarrow\partial_t f\ {\rm in}\ L^{2}_{loc}(B_1\times(0,1)).
\end{align*}
Since each $f_k$ satisfies \eqref{a-vel}, $f=\lim_{k\rightarrow\infty}f_k$ also satisfies \eqref{a-vel} and the proof is completed. 
\end{proof}

\begin{proof}[Proof of Theorem \ref{main1} with $\partial_t f\in L^2$]
The idea of the proof is similar to \cite[Section 2.1]{Ton-b} for the smooth case.
First we prove that $f$ is a strong solution of $v=h+u^\perp$. By \eqref{Bineq} and \eqref{a-vel}, it holds for any  $\psi\in C^1_c(B_1\times \mathbb R \times (0,1); \R^+)$ that
\begin{equation}\label{a-vel2}
    0\le\int_{0}^{1} \int_{M_t} (\nabla \psi -\psi h)\cdot \{h+(u\cdot \nu)\nu-v\} \, d\mathcal H^{n-1} dt.
\end{equation}
We use the Lebesgue differentiation theorem with respect to the parabolic ball. The following
can be proved by adapting the proof in \cite[Chapter 7]{Rudin} for parabolic balls and $L^2$ norm: 
For $g\in L^2(\mathbb R^{n-1}\times
\mathbb R)$ and $P_r(y,s):=\{(x,t)\in\mathbb R^{n-1}\times \mathbb R\,:\, \max\{|x-y|,|t-s|^{1/2}\}<r\}$, for $\mathcal L^n$-a.e.~$(y,s)
\in\mathbb R^{n-1}\times\mathbb R$, we have
$$\lim_{r\rightarrow 0}\frac{1}{\mathcal L^{n}(P_r(y,s))}\int_{P_r(y,s)}|g(y,s)-g(x,t)|^2\,dxdt=0.$$
We note that the well-known version of Lebesgue differentiation theorem does not 
cover the use of the parabolic balls since they are not ``nicely shrinking sets'' in \cite[7.9]{Rudin},
but the same proof using $P_r(y,s)$ works.
Using this, let $(y,s)$ be a Lebesgue point of the functions $\partial_t f$, $\nabla^2 f$ and $u$,
which is true $\mathcal L^n$-a.e.~on $B_1\times(0,1)$. For any $\tilde\psi\in
C_c^1(\mathbb R^n;\mathbb R^+)$ and $\eta\in C_c^1((-1,1);[0,1])$ with $\int_{\mathbb R}\eta(t)\,dt=1$, define $\psi(x_1,\ldots,x_n,t):=\tilde\psi(x_1,\ldots,x_n)\eta(t)$. Then for all small $\lambda>0$, define
$$\psi_\lambda(x_1,\ldots,x_n,t)=\lambda^{-n} \psi \Big(\frac{x_1-y_1}{\lambda},\ldots,\frac{x_{n-1}-y_{n-1}}{\lambda},\frac{x_n-f(y,s)}{\lambda}, \frac{t-s}{\lambda^2}\Big).$$ 
Here, the reason that we use $\lambda^{-n}$ and not $\lambda^{-n-1}$ is that, in the following,
we want $\iint_{P_\lambda}\psi_\lambda\rightarrow 0$ while $\iint_{P_\lambda}\nabla\psi_\lambda=O(1)$ as $\lambda\rightarrow 0$. 
If ${\rm spt}\,\tilde\psi\subset B^{n}_R$ for $R\geq 1$, we have ${\rm spt}\,\psi_\lambda
\subset B_{\lambda R}^{n}((y,f(y,s)))\times B^1_{\lambda^2}(s)$,
thus for all sufficiently small $\lambda$, $\psi_\lambda\in C_c^1(B_1\times\mathbb R\times(0,1);\mathbb R^+)$.  We estimate as 
\begin{equation}\label{a-vel3}
\begin{split}
\Big|&\int_0^{1}\int_{M_t}  \psi_\lambda\,h\cdot\{h+(u\cdot\nu)\nu-v\}\,d\mathcal H^{n-1}dt\Big| \\
&\leq \lambda^{-n}\sup|\psi|\,\iint_{B_{\lambda R}(y)\times B_{\lambda^2}^1(s) } (|h|^2+|h||u|+|h||v|)\sqrt{1+
|\nabla f|^2}\,dxdt \\
&\leq \frac{2\lambda\omega_{n-1}\sup\,|\psi| R^{n+1}(1+\sup\,|\nabla f|)}{\mathcal L^n(P_{\lambda R}(y,s))}
\iint_{P_{\lambda R}(y,s)}  (|h|^2+|h||u|+|h||v|)\,dxdt,
\end{split}
\end{equation}
where we note $\mathcal L^n(P_{\lambda R}(y,s))=2\omega_{n-1}(\lambda R)^{n+1}$. Since $(y,s)$
is a Lebesgue point, \eqref{a-vel3} shows 
\begin{equation}\label{a-vel4}
\lim_{\lambda\rightarrow 0}\int_0^{1}\int_{M_t}  \psi_\lambda\,h\cdot\{h+(u\cdot\nu)\nu-v\}\,d\mathcal H^{n-1}dt=0.
\end{equation}
Next, we estimate the term involving $\nabla\psi_\lambda$ in \eqref{a-vel2}. To do so, 
we estimate the difference of $\nabla\psi_\lambda$ evaluated on $M_t$ and on the tangent plane of $M_t$
at $(y,s)$. For $(x,t)\in B_{\lambda R}(y)\times B_{\lambda^2}^1(s)$ and by \eqref{lip}, we have
\begin{equation}\label{a-vel5}
|f(x,t)-f(y,s)-\nabla f(y,s)\cdot(x-y)|\leq [f]_{C^{1,\alpha}}(1+R^{1+\alpha})\lambda^{1+\alpha}.
\end{equation}
Writing $\tilde f(x):=f(y,s)+\nabla f(y,s)\cdot(x-y)$, we can use \eqref{a-vel5} to estimate
\begin{equation}\label{a-vel6}
|\nabla\psi_\lambda(x,f(x,t),t)-\nabla\psi_{\lambda}(x,\tilde f(x),t)|\leq \lambda^{-n-1+\alpha}
\|\nabla^2\psi\|_{C^0}[f]_{C^{1,\alpha}}(1+R^{1+\alpha}).
\end{equation}
Write $w:=\{h+(u\cdot\nu)\nu-v\}$ (evaluated on $M_t$) and using \eqref{a-vel6},
\begin{equation}\label{a-vel7}
\begin{split}
\Big|\int_0^{1}&\int_{M_t}\nabla\psi_\lambda\cdot w\,d\mathcal H^{n-1}dt-\iint_{B_{\lambda R}(y)
\times B_{\lambda^2}^1(s)} \nabla\psi_{\lambda}(x,\tilde f(x),t)\cdot w\sqrt{1+|\nabla f|^2}\,dxdt\Big| \\ 
&\leq \lambda^{-n-1+\alpha}c(\alpha,R,\|\psi\|_{C^2},\|f\|_{C^{1,\alpha}})\iint_{B_{\lambda R}(y)
\times B_{\lambda^2}^1(s)} |w|\,dxdt \\
&\leq  \frac{\lambda^\alpha c(\alpha,R,\|\psi\|_{C^2},\|f\|_{C^{1,\alpha}})}{
\mathcal L^n(P_{\lambda R}(y,s))}\iint_{P_{\lambda R}(y,s)} |w|\,dxdt 
\rightarrow 0\,\,\,(\lambda\rightarrow 0),
\end{split}
\end{equation}
where we used the fact that $(y,s)$ is a Lebesgue point of $w$. Let $\tilde x:=(x-y)/\lambda$ and
$\tilde t:=(t-s)/\lambda^2$ and we see that (writing $\tilde w(\tilde x,\tilde t):=w(x,t)$)
\begin{equation}\label{a-vel8}
\begin{split}
&\iint_{B_{\lambda R}(y)
\times B_{\lambda^2}^1(s)} \nabla\psi_\lambda(x,\tilde f(x),t)\cdot w\sqrt{1+|\nabla f|^2}\,dxdt \\
&=\iint_{B_R\times (-1,1)}\eta(\tilde t)\{\nabla\tilde\psi(\tilde x,\nabla f(y,s)\cdot \tilde x)
\cdot\tilde w\}(\sqrt{1+{|\nabla f(y,s)|^2}}+O(\lambda^{\alpha}))\,d\tilde x d\tilde t,
\end{split}
\end{equation}
where we again used \eqref{lip}. Because of the property of the Lebesgue point, the last quantity converges to (note $\int_{-1}^1\eta\,dt=1$)
\begin{equation}\label{a-vel9}
\int_{B_R}\nabla\tilde\psi(\tilde x,\nabla f(y,s)\cdot \tilde x)\cdot w(y,s)\sqrt{1+|\nabla f(y,s)|^2}
\,d\tilde x=\Big(\int_{{\rm Tan}_{Y} M_s}\nabla\tilde \psi\,d\mathcal H^{n-1}\Big)\cdot w(y,s),
\end{equation}
where ${\rm Tan}_Y M_s$ denotes the tangent space of $M_s$ at $Y=(y,f(y,s))$. 
Combining \eqref{a-vel2}, \eqref{a-vel4} and \eqref{a-vel7}-\eqref{a-vel9}, we obtain
\begin{equation}
0\leq \Big(\int_{{\rm Tan}_Y M_s}\nabla\tilde \psi\,d\mathcal H^{n-1}\Big)\cdot \{h+(u\cdot\nu)\nu-v\}(y,s)
\end{equation}
for any $\tilde \psi\in C_c^1(\mathbb R^n;\mathbb R^+)$. By integration by parts, the integral is 
perpendicular to ${\rm Tan}_Y M_s$. We may still choose a non-negative $\tilde \psi$ so that
the integral is equal to $\pm \nu(y,s)$, and since $\{h+(u\cdot\nu)\nu-v\}(y,s)$ is parallel to $\nu(y,s)$, 
it has to be $0$. This ends the proof that $v=h+(u\cdot\nu)\nu$ $\mathcal L^n$-a.e.~on
$B_1\times (0,1)$. 

Once we establish the PDE, the regularity of the solution is standard. For the completeness,
we present the proof. Fix $s\in (0,1)$ and let $\zeta\in C^{\infty}(B_1\times(0,1))$ be a non-negative cutoff function such that it vanishes on the parabolic boundary and 
$\zeta=1$ on $B_{1-s}\times[s^2,1)$. Then $\tilde{f}=\zeta\,f$ satisfies
\begin{equation}\label{eq:cutoff eq}
    \partial_t \tilde{f}-\sum_{i,j=1}^{n-1}a_{ij}(x,t)\partial^2_{x_i x_j}\tilde{f}=F(x,t)\ \ \mbox{ for\ $\mathcal L^n$-a.e.}\ (x,t)\in B_1\times(0,1).
\end{equation}
Here the coefficients $a_{ij}=\delta_{ij}-\frac{\partial_{x_i}f \partial_{x_j}f}{1+\lvert \nabla f\rvert^2}$ are uniformly elliptic and H$\ddot{\rm o}$lder continuous and we put
\[
F=\zeta u\cdot \nu\sqrt{1+\lvert \nabla f\rvert^2}+f\big(\partial_t\zeta-\sum_{i,j=1}^{n-1}a_{ij}\partial^2_{x_i x_j}\zeta\big)-2\sum_{i,j=1}^{n-1}a_{ij}\partial_{x_i}f\partial_{x_j}\zeta.
\]
By \eqref{lip} and \eqref{hdef3}, we have
$F\in L^q((0,1);L^p(B_1))$. 
Thus it holds from \cite[Theorem 7.3.9]{Kr08} that there exists a unique solution $f_1$
of \eqref{eq:cutoff eq} with $f_1(x,0)=0$ such that
\[
f_1\in W^{1,q}((0,1);L^p(B_{1}))\cap L^q((0,1);W^{2,p}(B_1)\cap W^{1,p}_0(B_1)).
\] 
Since both $f_1$ and $\tilde{f}$ are also the unique solution of \eqref{eq:cutoff eq} 
starting from $0$ in 
\[
W^{1,2}((0,1);L^2(B_1))\cap L^2((0,1);W^{2,2}(B_1)\cap W^{1,2}_0(B_1)),
\] 
it holds that $f_1=\tilde f$, and since $\tilde f=f$ on $B_{1-s}\times[s^2,1)$,
the proof is completed. 
\end{proof}
\section{Proof of Theorem \ref{main1}: the general case}\label{ver2}
In this section, we assume that $\partial_t f$ is a signed Radon measure. 
We first observe that we may regard 
the right-hand side of \eqref{Bineq} as a positive operator defined on $C_c^1((B_1\times\mathbb R)
\times(0,1))$. Then, it is well-known (see, for example, \cite[Corollary 1.8.1]{EvGa92})
that it extends uniquely 
to a nonnegative bounded linear operator on $C_c((B_1\times\mathbb R)\times(0,1))$. 
By the Riesz representation theorem, there exists a Radon measure $\xi$ whose support is
contained in $\cup_{t\in(0,1)}(M_t\times\{t\})$ and
\begin{equation}\label{ra1}
\int_0^{1}\int_{M_t}(\nabla\phi-\phi\,h)\cdot\{h+(u\cdot\nu)\nu\}+\partial_t\phi\,
d\mathcal H^{n-1}dt=\int_{\cup_{t\in(0,1)}(M_t\times\{t\})} \phi\,d\xi
\end{equation}
for all $\phi\in C_c^1((B_1\times\mathbb R)
\times(0,1))$. Here, note that $\phi$ need not be non-negative. 

Before we present the rigorous proof, we give a formal proof. 
Suppose for a moment that $f$ is smooth
and consider the signed distance function $d(\cdot,t):=\pm {\rm dist}\,(\cdot,M_t)$ 
depending above/below of $M_t$ in $B_1\times\mathbb R$. Given $\phi\in C_c^1(B_1\times
(0,1))$, we use $\phi\, d$ in \eqref{ra1}. Since $d=0$ on $M_t$, the right-hand side
is $0$. Moreover, $\partial_t(\phi\,d)=\phi\,\partial_t d=-\phi\frac{\partial_t f}{\sqrt{1+|\nabla f|^2}}$ and $\nabla(\phi\,d)=\phi\,\nu$ on $M_t$. Plugging these in, we obtain
$$
\int_0^1\int_{B_1}\phi\,\nu\cdot\{h+(u\cdot\nu)\nu\}\sqrt{1+|\nabla f|^2}-\phi\,\partial_t f\,dxdt=0
$$
and 
by the arbitrariness of $\phi$, we see
that the motion law \eqref{m1} is satisfied. Since the distance function is not 
smooth for $M_t$ in general, this is a formal argument, but we show that 
this approach works if $f$ is approximated properly with careful estimates on the
errors. 

Let $\rho(x,t)$ defined on $\mathbb R^{n}$ be the standard radially symmetric
mollifier with ${\rm spt}\,\rho\subset B_1^n$ 
and define $\rho^{\varepsilon}(x,t):=\varepsilon^{-n-1}
\rho(x/\varepsilon,t/\varepsilon^2)$ for $\varepsilon>0$ so that $\int_{\mathbb R^{n-1}\times\mathbb R} \rho^{\varepsilon}\,
dxdt=1$. Define
\begin{equation}
f^\varepsilon(x,t):=(\rho^\varepsilon\ast f)(x,t),\,\,\,\, M^\varepsilon_t:=\{(x,f^\varepsilon(x,t))\in \mathbb R^n\,:\, x\in B_{1-\eps}\},
\end{equation}
\begin{equation}
\tilde d^\varepsilon(X,t):=\left\{\begin{array}{ll} {\rm dist}\,(X,M^\varepsilon_t)
& \mbox{ if }x_n\geq f^\varepsilon(x_1,\ldots,x_{n-1},t),\\
-{\rm dist}\,(X,M^\varepsilon_t) & \mbox{ if } x_n<f^\varepsilon(x_1,\ldots,x_{n-1},t),
\end{array}\right.
\end{equation}
where $X=(x_1,\ldots,x_n)\in\mathbb R^n$. Let us fix a function $\eta^\eps\in C^\infty_c(\mathbb R;[-2\eps,2\eps])$
such that
$$
\eta^\eps(s)=s\,\,\mbox{ if }|s|\leq \eps,\,\,\,\eta^\eps(s)=0\,\,\mbox{ if }|s|\geq 2\eps
$$
and define
\begin{equation}
d^\eps(X,t):=\eta^\eps(\tilde d^\eps(X,t)).
\end{equation}
In the following, we fix $\phi\in C_c^1(B_1\times(0,1))$
and use $\phi\,d^\eps$ in \eqref{ra1},
so we are interested in the values of $d^\varepsilon$, $\nabla d^\varepsilon$, $\nabla^2 d^\varepsilon$
and $\partial_t d^\varepsilon$
on $M_t$. We regard $\phi$ to be a function defined on $B_1\times\mathbb R\times (0,1)$ 
which is independent of $x_n$ direction as well as a function on $B_1\times(0,1)$.
First, we show that $d^\eps$ is $C^\infty$ on a small neighborhood of $M_t$. The following
two lemmas are easy but we include the proofs for the reader's convenience:
\begin{lemma}\label{near1}
For $\varepsilon\in (0,1)$, we have
\begin{equation}\label{ra2}
\sup_{(x,t)\in B_{1-\varepsilon}\times(\varepsilon^2,1-\varepsilon^2)}|f^\varepsilon(x,t)-f(x,t)|\leq 2[f]_{C^{1,\alpha}}\varepsilon^{1+\alpha}.
\end{equation}
\end{lemma}
\begin{proof}
For $(y,s)\in B_\varepsilon\times (-\varepsilon^2,\varepsilon^2)$ and $(x,t)\in B_{1-\varepsilon}\times(\varepsilon^2,1-\varepsilon^2)$, by \eqref{lip}, 
we have
\begin{equation}\label{ra3}
|f(x+y,t+s)-f(x,t)-\nabla f(x,t)\cdot y|\leq 2[f]_{C^{1,\alpha}}\varepsilon^{1+\alpha}.
\end{equation}
Since $\int_{B_\varepsilon}y \rho^\varepsilon(y,s)\,dy=0$, we have
\begin{equation}\label{ra4}
f^\varepsilon(x,t)-f(x,t)= \iint_{B_\varepsilon \times(-\varepsilon^2,\varepsilon^2)}
\rho^{\varepsilon}(y,s)(f(x+y,t+s)-f(x,t)-\nabla f(x,t)\cdot y)\,dyds.
\end{equation}
Then, \eqref{ra3} and \eqref{ra4} give \eqref{ra2}. 
\end{proof}
\begin{lemma} \label{near2}
For $\varepsilon\in(0,1)$, we have
\begin{equation}\label{ra6}
\sup_{(x,t)\in B_{1-\varepsilon}\times(\varepsilon^2,1-\varepsilon^2)}\Big|\nabla^2 f^\varepsilon(x,t)\Big|
\leq c(\rho)\varepsilon^{\alpha-1}.
\end{equation}
\end{lemma}
\begin{proof}
For $i,j=1,\ldots,n-1$, by the symmetry of $\rho^\varepsilon$, we have
\begin{equation}\label{ra5}
\begin{split}
\partial^2_{x_i x_j}f^\varepsilon(x,t)&=\iint \partial_{x_i}\rho^\varepsilon(x-y,t-s)\partial_{y_j}f(y,s)\,dyds\\
&=\iint \partial_{x_i}\rho^\varepsilon(x-y,t-s)\big(\partial_{y_j}f(y,s)-\partial_{y_j}f(x,t)\big)\,dyds.
\end{split}
\end{equation}
Since $\iint |\nabla\rho^\varepsilon|\leq c(\rho)\varepsilon^{-1}$, \eqref{lip} and \eqref{ra5} show \eqref{ra6}. 
\end{proof}
By \eqref{ra2}, for all sufficiently 
small $\eps$ and for any $X\in M_t\cap{\rm spt}\,\phi(\cdot,t)$, 
${\rm dist}\,(X,M_t^\eps)\leq 2[f]_{C^{1,\alpha}}\eps^{1+\alpha}<\eps$.
Thus we have 
\begin{equation}
d^\eps(X,t)=\tilde d^\eps(X,t)\mbox{ for }X\in M_t\cap{\rm spt}\,\phi(\cdot,t).
\end{equation}
Also by computing the first and second fundamental
forms of the graph of $f^\eps$, one finds that the principal curvatures correspond to the solutions $\lambda$ of $${\rm det}\,\Big(\frac{\nabla^2 f^\eps}
{\sqrt{1+|\nabla f^\eps|^2}}-\lambda(I+\nabla f^\eps\otimes\nabla f^\eps)\Big)=0.$$
Thus the principal curvatures of $M_t^\eps$ are bounded by $\|\nabla^2 f^\eps\|\leq c(\rho)\eps^{\alpha-1}$ 
in particular by \eqref{ra6}. The signed distance function $\tilde d^\eps$ is then known to be smooth in the $\varepsilon^{1-\alpha}/c(\rho)$-neighborhood of $M_t^\eps$. Since $M_t\cap{\rm spt}\,\phi(\cdot,t)$
is in $2[f]_{C^{1,\alpha}}\eps^{1+\alpha}$-neighborhood of $M_t^\eps$, it is also contained there
and thus $d^\eps=\tilde d^\eps$ is smooth on $M_t\cap{\rm spt}\,\phi(\cdot,t)$. 
By construction, then, $\phi\,d^\eps\in C_c^1(B_1\times\mathbb R\times(0,1))$ and we may
justify using it in \eqref{ra1}. Since $d^\eps$ and $\nabla d^\eps$ respectively
converge to $0$ 
and $\nu$ uniformly on 
${\rm spt}\,\phi\cap \cup_{t\in(0,1)}(M_t\times\{t\})$, we can deduce that
\begin{equation}\label{ra7}
\begin{split}
0&=\lim_{\eps\rightarrow 0}\int_{\cup_{t\in(0,1)}(M_t\times\{t\})}\phi\,d^\eps\,d\xi\\
&=\lim_{\eps\rightarrow 0}\int_0^1\int_{M_t}\{\nabla(\phi\,d^\eps)-\phi\,d^\eps\,h\}\,\cdot\,\{
h+(u\,\cdot\,\nu)\,\nu\}+\partial_t(\phi\,d^\eps)\,d\mathcal H^{n-1}dt \\
&=
\int_0^1\int_{M_t}\phi\,\nu\,\cdot\{h+(u\,\cdot\,\nu)\nu\}\,d\mathcal H^{n-1}dt
+\lim_{\eps\rightarrow 0}\int_0^1\int_{M_t}\phi\,\partial_t d^\eps\,d\mathcal H^{n-1}dt.
\end{split}
\end{equation}
Here we prove
\begin{lemma}
\begin{equation}
\lim_{\eps\rightarrow 0}\int_0^1\int_{M_t}\phi\,\partial_t d^\eps\,d\mathcal H^{n-1}dt
=\int_0^1\int_{B_1}\partial_t\phi\,f\,dxdt.
\end{equation}
\end{lemma}
Once this is proved, with \eqref{ra7}, it implies that $f$ has the weak $L^2$ derivative
$\partial_t f=\{(h+u)\,\cdot\,\nu\}\sqrt{1+|\nabla f|^2}$ on $B_1\times(0,1)$, which is 
\eqref{m1}, and the argument for \eqref{m2} is the same as the previous section. 
\begin{proof}
We aim to change the domain of integration from $M_t$ to $M_t^\eps$ by using the nearest point
projection as follows. 
For each $X\in M_t\cap {\rm spt}\,\phi(\cdot,t)$, there exists a unique $X^*\in M_t^\eps$
such that ${\rm dist}\,(X,M_t^\eps)=|X-X^*|\leq 2[f]_{C^{1,\alpha}}\eps^{1+\alpha}$. 
By writing these two points as $X=(x,f(x,t))$ and $X^*=(x^*,f^\eps(x^*,t))$, 
note that these points are related by the following equation:
\begin{equation}\label{dist}
X^*+d^\eps(X,t)\nabla d^\eps(X^*,t)=X. 
\end{equation}
Since $X^*$ is the nearest point to $X$ in $M_t^\eps$, $X-X^*$
is perpendicular to $M_t^\eps$ at $X^*$ and the unit normal vector at $X^*$ is given by $\nabla d^\eps(X^*,t)$.
Then it is clear that \eqref{dist} holds. Since $\nabla d^\eps(X,t)=\nabla d^\eps(X^*,t)$,
we have
\begin{equation}\label{dist1}
X^*=X-d^\eps(X,t)\nabla d^\eps(X,t).
\end{equation}
We next claim that
\begin{equation}\label{dist2}
\partial_t d^\eps(X,t)=\partial_t d^\eps(X^*,t)=-\partial_t f^\eps(x^*,t)/\sqrt{1+|\nabla f^\eps
(x^*,t)|^2}.
\end{equation}
By differentiating 
$d^\eps(X-d^\eps(X,t)\nabla d^\eps(X,t),t)=0$ (which follows from $d^\eps(X^*,t)=0$) 
with respect to $t$, we obtain
\begin{equation}\label{dist3}
\nabla d^\eps(X^*,t)\cdot(-\partial_t d^\eps(X,t)\nabla d^\eps(X,t)-d^\eps(X,t)\partial_t\nabla
d^\eps(X,t))+\partial_t d^\eps(X^*,t)=0.
\end{equation}
Since $\nabla d^\eps(X^*,t)\cdot\nabla d^\eps(X,t)=1$ and $$\nabla d^\eps(X^*,t)\cdot\partial_t\nabla d^\eps
(X,t)=\nabla d^\eps(X,t)\cdot \partial_t\nabla d^\eps(X,t)=\frac12\partial_t|\nabla d^\eps(X,t)|^2=0,$$
we have the first equality of \eqref{dist2} from \eqref{dist3}. The second equality of \eqref{dist2}
is obtained by differentiating $d^\eps(x,f^\eps(x,t),t)=0$ (which follows from $(x,f^\eps(x,t))\in
M^\eps_t$) with respect to $t$ and by
using $\partial_{x_n} d^\eps(x,f^\eps(x,t),t)=1/\sqrt{1+|\nabla f^\eps|^2}$. 

Next, consider the map $F^\eps(\cdot,t)$ from $B_1\cap {\rm spt}\,\phi(\cdot,t)$ to $B_1$ defined by
\begin{equation}
x\longmapsto X=(x,f(x,t))\in M_t\longmapsto X^*=(x^*,f^\eps(x^*,t))\in M_t^\eps \longmapsto x^*=:F^\eps(x,t).
\end{equation}
\begin{figure}[tbh]
\centering
\includegraphics[width=10cm]{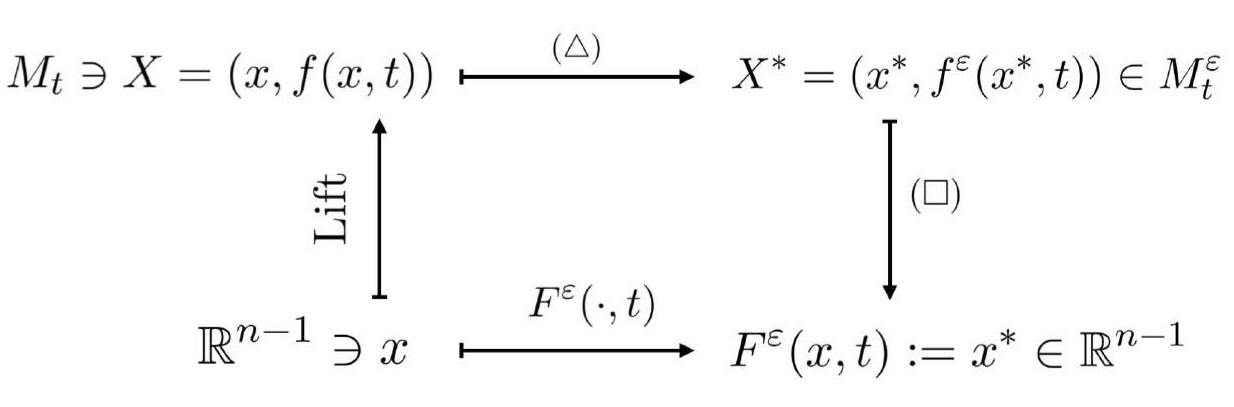}
\caption{}
\label{fig:diagram}
\end{figure}
As indicated in Figure 1, $x$ is lifted up to $M_t$ first,
and mapped to the nearest point on $M_t^\varepsilon$
(indicated by ($\triangle$)),
and then projected down to $\mathbb R^{n-1}$ (indicated by ($\square$)).
More explicitly, by projecting the equation \eqref{dist1} to $\mathbb R^{n-1}$, we have
\begin{equation}
F^\eps(x,t)=x^*=x-d^\eps(x,f(x,t),t)\nabla' d^\eps(x,f(x,t),t),
\end{equation}
where we recall that $\nabla'=(\partial_{x_1},\ldots,\partial_{x_{n-1}})$.
Since $d^\eps$ is $C^\infty$
as a function of $(X,t)$ near $M_t\cap{\rm spt}\,\phi(\cdot,t)$ and $f$ is $C^{1,\alpha}$, $F^\eps$ is also $C^{1,\alpha}((B_1\times(0,1))\cap {\rm spt}\,
\phi)$. We next compute $\nabla F^\eps(x,t)$ as
\begin{equation}\label{dist4}
\nabla F^\eps(x,t)=I-\nabla'\,d^\eps\otimes\nabla'\,d^\eps-\partial_{x_n}d^\eps\nabla f
\otimes\nabla'\,d^\eps-\big(\partial_{x_n}\nabla' \,d^\eps\otimes\nabla f+\nabla^{'2}\,
d^\eps\big)\,d^\eps.
\end{equation}
Since the principal curvatures of $M_t^\eps$ are bounded by $c\eps^{\alpha-1}$ and $M_t$ is
within $c\eps^{1+\alpha}$-neighborhood of $M_t^\eps$, $\nabla^2 d^\eps$ on $M_t$ is bounded by $c\eps^{\alpha-1}$
as well (see \cite[14.6]{GiTr83} for the expression of $\nabla^2 d^\eps$). Since $|d^\eps|\leq c\eps^{1+\alpha}$ on $M_t$, the last two terms of \eqref{dist4} involving $\nabla^2 d^\eps$ are bounded by $c\eps^{2\alpha}$ and vanish as $\eps\rightarrow 0$. To estimate the second and third terms of \eqref{dist4}, we compute
\begin{equation*}
\begin{split}
\Big|\nabla'\,d^\eps+&\partial_{x_n}\,d^\eps\nabla f\Big|=\Big|-\frac{\nabla f^\eps}{\sqrt{1+|\nabla f^\eps|^2}}\big|_{(x^*,t)}+\frac{1}{\sqrt{1+|\nabla f^\eps|^2}}\big|_{(x^*,t)}\nabla f(x,t)\Big| \\
&
\leq |\nabla f^\eps(x^*,t)-\nabla f(x,t)|\leq |\nabla f^\eps(x^*,t)-\nabla f(x^*,t)|+|\nabla
f(x^*,t)-\nabla f(x,t)|
\end{split}
\end{equation*}
and since $|x-x^*|\leq c\eps^{1+\alpha}$, this also vanishes as $\eps\rightarrow 0$. Thus
we have
\begin{equation*}
\lim_{\eps\rightarrow 0}\nabla F^\eps(x,t)=I
\end{equation*}
uniformly on $(B_1\times (0,1))\cap{\rm spt}\,\phi$. By the Inverse Function Theorem, for all
sufficiently small $\eps$, $F^\eps(\cdot,t)$ has the $C^1$ inverse function $G^\eps(\cdot,t)\,:\,F^\eps(\cdot,t)({\rm spt}\,
\phi(\cdot,t))\rightarrow {\rm spt}\,\phi(\cdot,t)$. One can also check that $G^\eps$ is H\"{o}lder
continuous in the direction of $t$ and 
\begin{equation}\label{dist5}
\lim_{\eps\rightarrow 0}\nabla G^\eps(x^*,t)=I.
\end{equation}
Now we compute (evaluating $\phi$ and $d^\eps$ at $(x,f(x,t),t)$) using \eqref{dist2}
\begin{equation}\label{dist6}
\begin{split}
\int_0^1\int_{M_t}&\phi\,\partial_t d^\eps\,d\mathcal H^{n-1}dt=\iint_{B_1\times(0,1)}
\phi\,\partial_t d^\eps \sqrt{1+|\nabla f|^2}\,dxdt \\
&=-\iint_{B_1\times(0,1)}\phi\,\partial_t f^\eps(F^\eps(x,t),t)\frac{\sqrt{1+|\nabla f(x,t)|^2}}{\sqrt{1+|\nabla f^\eps(F^\eps(
x,t),t)|^2}}\,dxdt.
\end{split}
\end{equation}
By our assumption, $\partial_t f$ is a signed Radon measure, so that the functional 
$$\mathcal{R}[\varphi]:=-\iint_{B_1\times(0,1)} \partial_t\varphi\, fdxdt$$ 
defined for $\varphi\in C^1_c(B_1\times(0,1))$ is continuous with respect to 
the topology of uniform convergence. We also have
\[
\partial_t f^{\varepsilon}(z,t)=\iint_{B_1\times(0,1)} \partial_t\rho^{\varepsilon}(z-y,t-s) f(y,s)\,dyds=\mathcal{R}[\rho^{\varepsilon}(\cdot-z,\cdot-t)]
\]
and the Fubini theorem shows (the domain of integration $B_1\times(0,1)$ omitted)
\[
\iint \varphi(x,t)\,\mathcal{R}[\rho^{\varepsilon}(\cdot-F^\eps(x,t),\cdot-t)]\,dxdt=\mathcal{R}\left[\iint \varphi(x,t)\rho^{\varepsilon}(\cdot-F^\eps(x,t).\cdot-t)\,dxdt\right].
\]
Using these, we see that
\begin{equation}\label{dist7}
{\small \begin{split}
&-\iint  \phi\, \partial_t f^{\varepsilon}(F^{\varepsilon}(x,t),t) \sqrt{\frac{1+|\nabla f(x,t)|^2}{1+|\nabla f^{\varepsilon}(F^{\varepsilon}(x,t),t)|^2}}\,dxdt
\\&=-\iint \phi\, \mathcal R[\rho^\eps(\cdot-F^\eps(x,t),\cdot-t)]\sqrt{\frac{1+|\nabla f(x,t)|^2}{1+|\nabla f^{\varepsilon}(F^{\varepsilon}(x,t),t)|^2}}\,dxdt
\\&=-\mathcal{R}\left[\iint \rho^{\varepsilon}(\cdot-F^{\varepsilon}(x,t),\cdot-t)\phi(x,t)\sqrt{\frac{1+|\nabla f(x,t)|^2}{1+|\nabla f^{\varepsilon}(F^{\varepsilon}(x,t),t)|^2}}\,dxdt\right]\\&=-\mathcal{R}\left[\iint \rho^{\varepsilon}(\cdot-x^*,\cdot-t)\phi(G^{\varepsilon}(x^*,t),t)\sqrt{\frac{1+|\nabla f(G^{\varepsilon}(x^*,t),t)|^2}{1+|\nabla f^{\varepsilon}(x^*,t)|^2}}|\mathrm{det}\nabla G^{\varepsilon}|dx^*dt\right].
\end{split}}
\end{equation}
In the last line, we changed variables $x=G^\eps(x^*,t)$. 
It is clear form \eqref{dist5} that the function 
\[
(y,s)\mapsto \iint \rho^{\varepsilon}(y-x^*,s-t)\phi(G^{\varepsilon}(x^*,t),t)\sqrt{\frac{1+|\nabla f(G^{\varepsilon}(x^*,t),t)|^2}{1+|\nabla f^{\varepsilon}(x^*,t)|^2}}|\mathrm{det}\nabla G^{\varepsilon}|\,dx^*dt
\]
converges to $\phi(y,s)$ as $\varepsilon\rightarrow0$ uniformly for $(y,s)\in B_1\times(0,1)$. Therefore \eqref{dist6} and \eqref{dist7} show
\[
\lim_{\eps\rightarrow 0}\int_0^1\int_{M_t}\phi\,\partial_t d^\eps\,d\mathcal H^{n-1}dt=-\mathcal R
[\phi]
\]
and the proof is completed.
\end{proof}
\section{Proof of Theorem \ref{th:existence}}\label{PDE}
Since the short-time existence of the $C^{1,\alpha}$ solution 
is already established in \cite[Theorem 2.5(1)(3)]{TaTo16}, 
we only need to consider the situation described in Section
\ref{setting} after a suitable change of variables. To apply 
Theorem \ref{main1}, the only missing piece is that $\partial_t f$ 
is a signed Radon measure. As we alluded in the introduction, we
show that $\partial_t f\in L^2$ in this case of the limit of 
Allen-Cahn equation with a transport term. In \cite{TaTo16}, 
the method of the proof is to approximate $v=h+u^\perp$ by the Allen--Cahn equation with a transport term coming from $u$ as follows. 
With an appropriate initial datum $\varphi^\eps|_{t=0}$ 
derived from $M_0$, one solves
\begin{equation}\label{eq:Allen-Cahn}
    \partial_t\varphi^{\varepsilon}+u^{\varepsilon}\cdot\nabla\varphi^{\varepsilon}=\Delta\varphi^{\varepsilon}-\frac{W'(\varphi^{\varepsilon})}{\varepsilon^2}\ \ {\rm in}\ \ \Omega\times(0,\infty),
\end{equation}
where $\varepsilon>0$ is a small parameter tending to $0$ and $u^{\varepsilon}$ is a smooth approximation of $u$. Moreover $W\in C^3(\R;\mathbb R^+)$ satisfies
\[W(\pm1)=0,\ \ W'<0\ {\rm on}\ (\gamma,1)\ {\rm and}\ W'>0\ {\rm on}\ (-1,\gamma)\ {\rm for\ some}\ \gamma\in(-1,1)\]
and
\[W''(x)\ge\kappa\ {\rm for\ all}\ 1\ge|x|\ge\alpha\ {\rm for\ some}\ \alpha\in(0,1),\ \kappa>0.\]
Define a constant $\sigma:=\int_{-1}^1\sqrt{2W(s)}\,ds$. Then it is 
proved that there exists a sequence $\varepsilon_i\downarrow0$ such that
\[\mu^{\varepsilon_i}_t:=\frac{1}{\sigma}\left(\frac{\varepsilon_i|\nabla\varphi^{\varepsilon_i}(x,t)|^2}{2}+\frac{W(\varphi^{\varepsilon_i}(x,t))}{\varepsilon_i}\right)\,dx\rightarrow\mu_t\ \ {\rm as}\ i\rightarrow\infty\]
for all $t\geq 0$ as Radon measures and that
$\mu_t$ is rectifiable and integral measure for $\mathcal L^1$-a.e.~$t\geq 0$.
The measure $\mu_t$ induces naturally a unique integral varifold $V_t$ for $\mathcal L^1$-a.e.~$t\ge0$ and it is proved that $\{V_t\}_{t\geq 0}$ is a weak solution 
of $v=h+u^\perp$ (see \cite[Theorem 2.2]{TaTo16} for the precise
statement). It is also proved (see \cite[Proposition 8.4]{TaTo16}) that there exists a limit ``phase function''
$w\in BV_{loc}(\Omega\times[0,\infty))\cap C^{\frac{1}{2}}_{loc}([0,\infty);L^1(\Omega))$ such that
\[w^{\varepsilon_i}:=\Phi\circ\varphi^{\varepsilon_i}(\cdot,t)\rightarrow w(\cdot,t)\ \ {\rm in}\ L^1_{loc}(\Omega)\ \ {\rm as}\ \ i\rightarrow\infty\ \ {\rm for\ all}\ t\ge0.\]
Here $\Phi(s):=\frac{1}{\sigma}\int_{-1}^s\sqrt{2W(y)}\,dy$.  
Using the estimates in the proof of \cite[Theorem 2.2]{TaTo16}, one can 
prove the following.
\begin{lemma}\label{lm:t-derivative}
Let $\varphi^{\varepsilon}$ be the solution of \eqref{eq:Allen-Cahn}
constructed in \cite{TaTo16}. Then there exists a constant $C$ independent of $\varepsilon>0$ such that $w^{\varepsilon}=\Phi\circ\varphi^{\varepsilon}$ satisfies 
\[ \int_0^T\int_{\Omega}|\psi\, \partial_t w^{\varepsilon}|\,dxdt\le C\left(\int_0^T\int_{\Omega}\psi^2\,d\mu_t^{\varepsilon}dt\right)^{\frac{1}{2}}\]
for all $\psi\in C_c (\Omega \times (0,T))$.
\end{lemma}

\begin{proof} By the Cauchy-Schwarz inequality, we have
\begin{align*}
\int^T_0 \int_{\Omega} |\psi \,\partial_t w^{\varepsilon} |\, dxdt  &=\int^{T}_0 \int_{\Omega}  |\psi|\, \sigma^{-1} \sqrt{2W(\varphi^{\varepsilon})} \, |\partial _t \varphi ^{\varepsilon}|\, dxdt  \\
    & \le 2\sigma^{-1} \left(\int^T_0 \int_{\Omega}\psi ^2 \frac{W(\varphi^{\varepsilon})}{\varepsilon} \,dxdt\right)^{\frac{1}{2}}\left(\int^T_0 \int_{\Omega}\frac{\varepsilon(\partial_t\varphi^{\varepsilon})^2}{2}\,dxdt\right)^{\frac{1}{2}}.
\end{align*}
As in the proof of \cite[Theorem 2.2]{TaTo16}, the second
term of the right-hand side is bounded by a constant independent of $\eps$
and the first term is bounded by $(\int_0^T\int_\Omega\,\psi^2\,d\mu_t^\eps dt)^{1/2}$, which shows the desired inequality.
\end{proof}

\begin{lemma}\label{lm:velocity formula}
Let $\{\mu_t\}_{t\ge0}$ and $w$ be a family of measures and a limit phase function as above. Then there exists ${\rm v}\in L^2_{loc}(d\mu_t dt)$ such that for any $T>0$ and $\psi\in C^1_c(\Omega\times(0,T))$,
\begin{equation}\label{eq:velocity formula}
    0=\int_{0}^{T}\int_{\Omega}\psi\, {\rm v}\,d\mu_tdt+\int_{0}^{T}\int_{\Omega}w\,\partial_t\psi\, dxdt.
\end{equation}
\end{lemma}
\begin{proof}
Let $w^i=w^{\varepsilon_i}$ be as above. From Lemma \ref{lm:t-derivative}, by extracting a subsequence, there exists a signed measure $\eta$ on $\Omega \times (0, T)$ such that
\begin{align*}
    \lim_{i\rightarrow\infty}\int_{\Omega\times(0,T)}\psi\, \partial_t w^i \,dx dt=\int_{\Omega\times(0,T)}\psi \,d\eta.
\end{align*}
Again by Lemma \ref{lm:t-derivative} and by $d\mu_t^{\varepsilon_i}dt \to d\mu_t dt$ as $i\to \infty$, 
\begin{equation*}
    \left \vert \int_{\Omega \times (0,T)}\psi \, d\eta \right \rvert \le C\left(\int_0^T\int_{\Omega}\psi^2\,d \mu_t(x) dt\right)^{\frac{1}{2}}.
\end{equation*} 
Thus there exists ${\rm v}\in L^2_{\mathrm{loc}}(d\mu_t dt)$ such that for any $\psi \in C^1_c(\Omega \times (0,T))$, 
\begin{equation*}
    \lim_{i\rightarrow\infty}\int_{\Omega\times(0,T)}\psi\, \partial_t w^i\, dx dt=\int_{\Omega \times (0,T)}\psi \, d\eta=\int_{\Omega \times (0,T)}\psi\, {\rm v}\, d\mu_t dt.
\end{equation*}
Integrating 
\begin{align*}
    \frac{d}{dt} \int_{\Omega} \psi\, w^i \, dx 
    &=\int _{\Omega} w^i\, \partial_t \psi \, dx +\int _{\Omega} \psi \,\partial_t w^i \, dx
\end{align*}
over $(0,T)$, we have
\begin{equation*}
    0=\int^{T}_{0} \int _{\Omega} w^i\, \partial_t \psi \, dxdt +\int^{T}_{0} \int _{\Omega} \psi\, \partial_t w^i \, dxdt.
\end{equation*}
Thus, letting $i\to \infty$, we have (\ref{eq:velocity formula}) and complete the proof.
\end{proof}

Now we are ready to finish the proof of Theorem \ref{th:existence}. Suppose that ${\rm spt}\,\mu_t$ is
locally represented as a $C^{1,\alpha}$ graph $f$ as in Section
\ref{setting} and assume without loss of generality that $-1< f< 1$
on $B_1\times(0,1)$. Thus, ${\rm spt}\,\mu_t=\{(x,f(x,t))\,:\, x\in B_1\}$
and $\{(x,y)\in B_1\times (-1,1)\,:\,y<f(x,t)\}=\{(x,y)\in B_1\times(-1,1)\,:\, w(x,y,t)=1\}$
for $t\in (0,1)$. Here, we implicitly use the fact established in
\cite{TaTo16} that the phase function $w$ has the boundary of $\{w=1\}$
on the support of $\mu_t$ (see \cite[Theorem 2.3(2)]{TaTo16}) and we assume without loss of generality that
$\{w=1\}$ lies below the graph here.
Let $\psi \in C^1_c(B_1 \times (0,1))$ be arbitrary,
and let $\omega(x_{n})$ be a function such that $\omega=1$ for $|x_n|<1$ and $\omega=0$ if $|x_n|>2$ and smooth otherwise with $|\omega|\leq 1$. 
Define $\tilde \psi(x_1,\ldots,x_n,t)=\omega(x_n)\psi(x_1,\ldots,x_{n-1},t)$ so that $\tilde\psi\in C_c^1(B_1\times\mathbb R\times(0,1))$.
By Lemma \ref{lm:velocity formula}, we have 
\begin{equation}\label{eq:halfway}
    0=\int_{0}^{1}\int_{B_1\times(-2,2)}\tilde \psi\, {\rm v} \, d\mu_tdt+\int_0^1\int_{\{\, (x,y) \in B_1\times(-2,2) \colon y \le f(x,t)\}}\partial_t \tilde{\psi} \, dxdt.
\end{equation}
Since $\omega=1$ on the support of $\mu_t$, the first term of the right-hand side of \eqref{eq:halfway} is equal to 
\[
\int_{0}^{1}\int_{B_1}\tilde{\psi}(x,f(x,t),t)\, {\rm v}\, \sqrt{1+\lvert \nabla f \rvert ^2}\, dxdt=\int_0^1\int_{B_1}\psi(x,t) \,{\rm v}\, \sqrt{1+\lvert \nabla f \rvert ^2}\, dxdt.
\]
The second term of the right-hand side of \eqref{eq:halfway} is equal to 
\begin{align*}
    \int_{0}^{1}\int_{B_1}\partial_t\psi(x,t)\int_{-2}^{f(x,t)}\omega(y)\,dy dxdt=\int_0^1\int_{B_1}f(x,t)\,\partial_t\,\psi(x,t)\,dxdt.
\end{align*}
Thus we obtain
\begin{equation*}
    0=\int_{0}^{1}\int_{B_1}\psi\, {\rm v} \, \sqrt{1+|\nabla f|^2}\, dx dt+\int_{0}^{1}\int_{B_1} f\, \partial_t\psi\, dxdt.
\end{equation*}
Hence, there exists the weak derivative $\partial_t f$ of $f$ and it holds that 
\[
\partial_t f={\rm v}\sqrt{1+\lvert \nabla f \rvert^2} \in L^2_{loc}(B_1\times (0,1))
\]
and the proof is completed.

\bibliography{references1}
\bibliographystyle{plain}

\end{document}